\documentclass{mrlart7}
  \def\volno{18}
  \def\yrno{2011}
  \def\issueno{00}
  \def\lpageno{100NN} %temporary numbers
\usepackage{amsfonts,amsmath,amssymb,color,amscd,amsthm}

\newtheorem{theorem}[equation]{Theorem}
\newtheorem{lemma}{Lemma}[section]
\newtheorem{corollary}[lemma]{Corollary}
\newtheorem{thm}[lemma]{Theorem}
\newtheorem{proposition}[lemma]{Proposition}

\newtheorem*{question*}{Question}
\newtheorem*{problem*}{Problem}

\theoremstyle{definition}

\theoremstyle{remark}
\newtheorem{remark}[lemma]{Remark}
\newtheorem{example}[lemma]{Example}

\newcommand\Sym{{\mathrm{Sym}}}

\newcommand\Q{{\mathbb Q}}
\newcommand\R{{\mathbb R}}
\newcommand\A{{\mathbb A}}
\newcommand\Z{{\mathbb Z}}

\newcommand\C{{\mathbb C}}
\newcommand\Spec{{\mathrm{Spec}}}
\newcommand\p{{\mathbb P}}
\newcommand\im{{\mathbf{i}}}

\newcommand\PGL{{\mathrm{PGL}}}

\newcommand\SL{{\mathrm{SL}}}
\newcommand\End{{\mathrm{End}}}
\newcommand\GL{{\mathrm{GL}}}
\newcommand\Aut{{\mathrm{Aut}}}
\newcommand\Diff{{\mathrm{Diff}}}

\newcommand\tr{\hbox to 1mm  {${}^t \!  $} }

\title{Extension of automorphisms of rational smooth affine curves}
\thanks{The authors gratefully acknowledge support by the Swiss National Science Foundation Grant  "Birational Geometry" PP00P2\_128422 /1 and by the French National Research Agency Grant "BirPol", ANR-11-JS01-004-01.}

\author{J\'er\'emy Blanc}
\address{J. Blanc, Universit\"{a}t Basel, Mathematisches Institut, Rheinsprung $21$, CH-$4051$ Basel, Switzerland.}
\email{jeremy.blanc@unibas.ch}
\urladdr{http://jones.math.unibas.ch/~blanc/}

\author{Jean-Philippe Furter}
\address{J.-P. Furter, Dpt. of Math., Univ. of La Rochelle, av. Cr\'epeau, 17000 La Rochelle, France}
\email{jpfurter@univ-lr.fr}
\urladdr{http://perso.univ-lr.fr/jpfurter/}

\author{Pierre-Marie Poloni}
\address{P.-M. Poloni, Universit\"{a}t Basel, Mathematisches Institut, Rheinsprung $21$, CH-$4051$ Basel, Switzerland.}
\email{pierre-marie.poloni@unibas.ch}
\urladdr{http://jones.math.unibas.ch/~poloni/}

\begin{document}
\maketitle
\begin{abstract}
We provide the existence, for every complex rational smooth affine curve~$\Gamma$, of  a linear action of~$\Aut(\Gamma)$ on the affine 3-dimensional space $\mathbb{A}^3$, together with a $\Aut(\Gamma)$-equivariant closed embedding of~$\Gamma$ into $\mathbb{A}^3$.
It is not possible to decrease the dimension of the target, the reason for this obstruction is also precisely described.
\end{abstract}

\centerline{}

\section{Introduction}

Throughout this article, all varieties are algebraic varieties defined over the field $\C$ of complex numbers. The affine (resp. projective) $n$-space is denoted by $\A^n$ (resp.~$\p^n$).

It is well known that any smooth affine variety $X$ of dimension $n$ admits a closed embedding into $\A^m$, when $m \geq 2 n+1$ \cite[Theorem 1]{Srinivas}. If moreover $m\ge 2n+2$, then, by a result of Nori, Srinivas and Kaliman (see \cite{Srinivas} and \cite{Kaliman}), any two closed embeddings $\iota,\iota'\colon X\to \A^m$ are equivalent in the sense that there exists $f\in \Aut(\A^m)$ such that $\iota'=f\circ\iota$ .

In particular, if $\iota\colon X\to \A^m$ is a closed embedding of a smooth affine variety of dimension $n$ into some affine space of dimension $m\ge 2n+2$, then it follows that every automorphism $\varphi$ of~$X$ extends to an automorphism of the ambient space $\A^m$, since the two embeddings $\iota\circ\varphi$ and $\iota$ are equivalent.

However, Derksen, Kutzschebauch and Winkelmann showed in \cite{Derksen-Kutzschebauch-Winkelmann} that it is not always possible to extend the group structure of~$\Aut(X)$, i.e.\ to find a closed embedding $\iota \colon X\to \A^m$  and an action of~$\Aut(X)$ on $\A^m$ that restricts on $X$ to the action of~$\Aut(X)$ on it. More precisely, they proved  that there does not exist, for any integer~$m$, any injective group homomorphism from
$\Aut ( \C^* \times \C^* )\cong\GL _2 (\Z) \ltimes  (\C^*)^2$ to the group  $\Diff ( \R^m)$ of diffeomorphisms of~$\R^m$.

Recall that, if $G$ is an algebraic group  acting on an affine variety $X$, then  $X$ admits a $G$-equivariant closed embedding into a finite dimensional $G$-module (see \cite[Proposition 1.12, p. 56]{Borel}). In particular, there exist,
for every smooth affine curve~$\Gamma$, a linear action of~$\Aut(\Gamma)$ on an affine space $\A^m$ and a $\Aut(\Gamma)$-equivariant closed embedding of~$\Gamma$ into $\A^m$. A natural question is then to find the smallest possible $m$.

\smallskip

In this article, we settle the case of rational smooth affine curves. In this setting, the proof of Borel only gives the embedding dimension $m=2\cdot|\Aut(\Gamma)|$, when the automorphism group $\Aut(\Gamma)$ is finite. However, our main result shows that it is already possible to obtain $m=3$:

\begin{theorem}\label{Thm1}
Every rational smooth affine curve $\Gamma$ admits an $\Aut (\Gamma )$-equivariant closed embedding into  the affine space $\A^3$. Furthermore, there exist such embeddings for which the action of~$\Aut (\Gamma )$ on $\A^3$ is linear.
\end{theorem}

It is easy to construct closed embeddings into the affine plane $\A^2$ for all rational smooth affine curves $\Gamma$. But it is of course not possible in general to ask for $\Aut(\Gamma )$-equivariant embeddings into $\A^2$. Indeed, there exist rational smooth affine curves  whose automorphism groups are isomorphic to the alternating group ${\mathfrak A}_4$, to ${\mathfrak A}_5$, or to the symmetric group ${\mathfrak S}_4$ (see Section~\ref{Sec:Explicit}) and it is well known that the group ${\mathfrak A}_4$ has no faithful representation of dimension two. Since all finite subgroups of~$\Aut(\A^2)$ are linearizable, it follows that we cannot embed equivariantly such a curve into the plane, even if we allow non linear actions on $\A^2$.

In fact, we establish stronger impossibility statements showing that it would be also too optimistic in general to look for closed embeddings into $\A^2$ in such a way that every single automorphism of the curve extends to an automorphism of the ambient space (see Corollary~\ref{Coro:NoExt}).

\begin{theorem}\label{Thm2}
There exist  rational smooth affine curves $\Gamma$ with $\Aut(\Gamma)\not=1$ and such that for every closed embedding of~$\Gamma$ into $\A^2$, the identity on $\Gamma$ is its only automorphism that extends to an automorphism of~$\A^2$.
\end{theorem}

Let us also emphasize that Theorem~\ref{Thm1} cannot be generalized to all smooth affine curves.
Actually, there even exist, for every natural number $n$, smooth affine curves~$\Gamma$ which do not admit any $\Aut(\Gamma)$-equivariant closed embedding into $\A^n$.

To see this, recall that every finite group $G$ is equal to the automorphism group of a smooth projective curve, and thus of an affine one \cite{Greenberg}, and take a smooth affine curve $\Gamma_n$ whose automorphism group is isomorphic to $\left(\Z/2\Z\right)^{n+1}$. Then, $\Gamma_n$ does not admit any $\Aut(\Gamma_n)$-equivariant embedding into $\A^n$, because $\left(\Z/2\Z\right)^{n+1}$ does not act faithfully on $\A^n$. Indeed, by Smith theory, the action of a finite $p$-group on $\A^n$ has always a fixed point (see e.g. \cite[Th. 7.11, p 145]{Bredon}, \cite[p. 204]{Petrie-Randall}, or \cite[Proposition~1]{Derksen-Kutzschebauch-Winkelmann}) and the induced  tangential (linear) representation at that fixed point should be faithful too (see e.g. \cite[Lemma 4]{Derksen-Kutzschebauch-Winkelmann}).

It would however be interesting to know what happens in the case of smooth affine curves of genus $1$. Sathaye proved in \cite{Sathaye} that such curves admit closed embeddings into~$\A^2$. Nevertheless, we do not know what is the minimal $m$ (if it exists) such that every smooth affine curve $\Gamma$ of genus $1$ admits an $\Aut(\Gamma)$-equivariant closed embedding into~$\A^m$.

\medskip

The article is organized as follows.

Section~\ref{Sec:planeEmb} concerns embeddings of rational smooth affine  curves into the affine plane. We give examples of automorphisms of such curves that do not extend, and prove Theorem~\ref{Thm2}  (see Corollary~\ref{Coro:NoExt}).

Section~\ref{SecPlanarEmbedd} is devoted to the study of embeddings of smooth rational curves into $\A^3$ whose images are contained in a hyperplane. We prove that they are all equivalent and thus that any two closed embeddings of a rational smooth affine curve into $\A^2$ become equivalent, when seen as embeddings in $\A^3$ (Proposition~\ref{Prop:planarEmb}). This answers a question of Bhatwadekar and Srinivas in this case.

In section~\ref{SecSL2} we realize every non-empty subset of~$\p^1$ that is invariant by a subgroup~$H$ of~$\Aut(\p^1)$ as the fixed-point set of  a $H$-equivariant endomorphism of~$\p^1$  (Corollary~\ref{Coro:ExistHequi}). This result is used in Section~\ref{Sec:EmbQuad} to prove Theorem~\ref{Thm1}  (see Theorem~\ref{thm:GroupXC3}). Explicit formulas are given in Section~\ref{Sec:Explicit}.

\section{Embeddings of rational smooth affine curves into the plane}\label{Sec:planeEmb}

Let us recall that every rational smooth affine curve $\Gamma$ is isomorphic to $\p^1\setminus \Lambda$, where $\Lambda$ is a finite set of~$r\ge 1$ points.

In particular, it admits a closed embedding into $\A^2$. Indeed, $\Gamma$ can also be seen as the complement in $\A^1$ of a finite number (possibly zero) of points and we can consider the closed embedding $\tau\colon \Gamma\to\A^2$ given by $x\mapsto (x,\frac{1}{P(x)})$, where $P\in\C[x]$ is a polynomial whose roots  are exactly the removed points. Note that the image of~$\tau$ is the curve of~$\A^2$ defined by the equation $P(x)y=1$.

Moreover, the automorphism group $\Aut(\Gamma)$ of the curve $\Gamma=\p^1\setminus \Lambda$ is equal to the group of automorphisms of~$\p^1$ that preserve the set $\Lambda$. This gives a group homomorphism from $\Aut(\Gamma)$ to the symmetric group $\Sym_r$. Note that this homomorphism is injective if and only if $r\ge 3$.

If $r$ is equal  to $1$ or $2$, then $\Gamma$ is isomorphic to $\A^1$ or $\A^1\setminus\{0\}$, and its automorphism group  is  $\C^* \ltimes\C$ or $\{ \pm 1 \} \ltimes\C^*$ respectively. If $r\ge 3$, then $\Aut(\Gamma)$ is a finite group.

The Abhyankar-Moh-Suzuki theorem claims that all closed embeddings of~$\A^1$ into~$\A^2$ are equivalent to the one given by $t\mapsto (t,0)$. This implies that every automorphism of an affine line embedded into $\A^2$ extends to an automorphism of the ambiant space. If $r \geq 2$ we can on the contrary construct embeddings of the curve~$\Gamma$ which do not have this property. Actually, we can choose embeddings such that, except the identity, no automorphisms of~$\Gamma$ extend.

\begin{lemma}
Let $\Gamma=\A^1\setminus \Delta$, where $\Delta$ is a non-empty finite set. Then, there exist infinitely many non-equivalent closed embeddings  $\iota\colon\Gamma\to\A^2$ such that the identity is the only automorphism of~$\A^2$ that preserves $\iota(\Gamma)$.
\end{lemma}

\begin{proof}
We can assume that $\Delta=\{0,a_1,\dots,a_m\}$, where $a_1,\dots,a_m\in \C\setminus\{0,1\}$, $m\ge 0$.
For every $k\geq2$, we denote by  $\iota_k\colon \Gamma\to \A^2$ the embedding given by
\[{x\mapsto \left(x,\frac{x-1}{x^k\prod_{i=1}^m (x-a_i)}\right).}\]
It  induces an isomorphism between $\Gamma$ and  the curve $\iota_k(\Gamma)$ defined by the equation  $$x=y x^k\prod_{i=1}^m (x-a_i)+1.$$

We first remark that any automorphism of~$\A^2$ that sends  $\iota_k(\Gamma)$ onto a curve of degree at most $\deg(\iota(\Gamma))=k+m+1$ is necessarily affine. Indeed, if  $f\colon (x,y)\mapsto (f_1(x,y),f_2(x,y))$ is the inverse of such an automorphism, we get:
$$\begin{array}{rcl}
\deg(f_1-f_2(f_1)^k \prod_{i=1}^m (f_1-a_i)-1)=(k+m) \deg f_1 + \deg f_2&\le&  k+m+1.\end{array}$$
This implies that $\deg(f_1)=\deg(f_2)=1$, i.e. that $f$ (and its inverse too) is affine.
In particular, all above embeddings are non-equivalent. We now show that the identity is the only affine automorphism of~$\A^2$ that preserves the curve $\iota_k(\Gamma)$.

Any such automorphism extends to an automorphism $\tau$ of~$\p^2$ preserving the line at infinity given by $z=0$ and the curve of equation
$$xz^{k+m}-yx^k \prod\limits_{i=1}^m (x-a_iz)-z^{k+m+1} =0.$$
On the line at infinity we get the two points  $[0:1:0]$ and $[1:0:0]$. The point $[1:0:0]$ is smooth with tangent $y=0$ and the point $[0:1:0]$ is singular with  tangent cone given by $x^k\prod_{i=1}^m(x-a_iz)=0$. Hence, both lines $x=0$ and $y=0$ are invariant. Therefore, $\tau$ is given by a diagonal automorphism of the form $[x:y:z]\mapsto [  \mu x: \nu y:z]$, $\mu,\nu\in \C^*$. Replacing in the equation yields $\mu=\nu=1$.
\end{proof}

The curves  $\A^1$ and $\A^1\setminus\{0\}$  admit  closed embeddings into $\A^2$ such that all their automorphisms  extend to  automorphisms  of~$\A^2$. Consider for example the curves of equations $y=0$ and  $xy=1$. However, it is no longer true for the curve $\A^1\setminus\{0,1\}$.

\begin{proposition}  \label{Prop:A101}
Let $\Gamma=\A^1 \setminus\{0,1\}$. For every closed embedding $\tau \colon \Gamma \to \A^2$, there exists an automorphism of~$\Gamma$ that does not extend to $\A^2$.
\end{proposition}

Before proving this statement, let us recall the following classical result (see e.g.~\cite[Theorem 2]{Furushima}).

\begin{lemma}\label{Lemm:FiniteGL2}
Every finite subgroup of~$\Aut(\A^2)$ is conjugate to a subgroup of~$\GL(2,\C)$.
\end{lemma}
\begin{proof}[Proof of Proposition~$\ref{Prop:A101}$]
Note that the group of automorphisms of~$\Gamma$ is the group $\Sym_3$ of permutations of a set of three elements, corresponding to the three points ``at infinity'', i.e.\ the points of~$\p^1\setminus \iota(\Gamma)$, where $\iota$ is any (open) embedding of~$\Gamma$ in $\p^1$. It is generated by the automorphisms $\rho\colon x\mapsto 1/(1-x)$ and $\sigma\colon x\mapsto 1-x$ and we have
\[\Aut(\Gamma)= \langle \sigma,\rho\ |\ \sigma^2=\rho^3=1, \sigma\rho\sigma^{-1}=\rho^{-1} \rangle=\Sym_3.\]
Suppose for contradiction that there exists a closed embedding $\tau\colon \Gamma\to \A^2$ for which every automorphism of~$\Gamma$ extends. Since the identity is the only automorphism of~$\A^2$ that restricts to the identity on a closed curve isomorphic to $\A^1 \setminus\{0,1\}$ (see Lemma~\ref{Lemma:fixed points of a plane polynomial automorphism} below), we would have a subgroup $G\subset \Aut(\A^2)$ isomorphic to $\Sym_3$ whose restriction to $\tau(\Gamma)$ yields $\Aut(\Gamma)$.

We now prove that this is impossible. First, recall that $G$ is conjugate to a subgroup of~$\GL(2,\C)$ (see Lemma~\ref{Lemm:FiniteGL2} above). Then, one easily checks that $G$ is conjugate to the subgroup $G'$ of~$\GL(2,\C)$ generated by
$$\hat{\rho}\colon (x,y)\mapsto (y,-x-y)\mbox{ and }\hat{\sigma}\colon (x,y)\mapsto (y,x).$$

We let $f\in\Aut(\A^2)$ be an automorphism such that $fGf^{-1}=G'$ and we consider the embedding $\hat{\tau}=f\circ\tau$ of~$\Gamma$ in $\A^2$. The automorphism group of~$\Gamma$ extends then to~$G'$ for this embedding.

Remark that the set $\{\omega\mid\omega^2 -  \omega+1=0\}\subset \Gamma$, which is the set of fixed points of~$\rho$, is an orbit of size $2$ of~$\Aut(\Gamma)$. But one checks that $G'\subset \GL(2,\C)$ does not have any orbit of size $2$ in the affine plane $\A^2$. This gives a contradiction.
\end{proof}

\begin{lemma} \label{Lemma:fixed points of a plane polynomial automorphism}
The set of fixed points of a plane polynomial automorphism is either a finite set of points $($possibly empty$)$,  a finite disjoint union of subvarieties isomorphic to $\A^1$, or the whole plane.
\end{lemma}

\begin{proof}
Using the amalgamated structure of~$\Aut(\A^2)$, it is observed in \cite{Friedland-Milnor} that a plane polynomial automorphism  is conjugate  either to a triangular automorphism $(x,y) \mapsto (ax + p(y), by+c)$ with $a,b,c \in \C$ and $p(y) \in \C[y]$, or  to some cyclically reduced element  (see \cite[I.1.3]{Serre}  or \cite[p. 70]{Friedland-Milnor} for the definition of a cyclically reduced element).
In the first case, an obvious computation shows that the set of fixed points is either empty,  a point, a finite disjoint union of subvarieties isomorphic to $\A^1$, or the whole plane.
In the second case, by \cite[Theorem 3.1]{Friedland-Milnor}, the set of fixed points is a non-empty finite set of points.
\end{proof}

Using tools of birational geometry, we can actually specify the statement of  Proposition~\ref{Prop:A101}. Indeed, Theorem~\ref{Thm:PtsInfinity} below shows that there is no closed embedding of the curve $\A^1 \setminus\{0,1\}$ into $\A^2$ such that the automorphism $\rho\colon x\mapsto 1/(1-x)$  extends to an automorphism of the affine plane.

Before we state this result, let us recall that any automorphism $f$ of~$\p^1$ of finite order $n>1$ is  conjugate to $[x:y]\mapsto [x:\xi y]$, where $\xi$ is a primitive $n$-th root of unity. In particular, it has the following properties:
\begin{enumerate}
\item
the automorphism $f$ fixes exactly two points of~$\p^1$;
\item
all other orbits under the action of~$f$ have size $n$.
\end{enumerate}
Thus, the following holds for every automorphism $g\in \Aut(\Gamma)$ of order $n>1$ of a rational smooth curve $\Gamma$.
\begin{enumerate}
\item
The automorphism $g$ fixes $0$, $1$ or $2$ points of~$\Gamma$;
\item
all other orbits under the action of~$g$ have size $n$.
\end{enumerate}

\begin{thm}\label{Thm:PtsInfinity}
Let $\Gamma$ be a rational smooth  affine curve  and let $g\in \Aut(\Gamma)$ be an automorphism.

\begin{enumerate}
\item
If $g$ fixes at most one point of~$\Gamma$, there is a closed embedding $\tau\colon \Gamma\to \A^2$ such that $g$ extends to an automorphism of~$\A^2$.
\item
If $g$ is of finite order $n>1$ with $n$ odd and if it fixes exactly two points of~$\Gamma$, then there is no closed embedding $\tau\colon \Gamma\to \A^2$ such that $g$ extends to an automorphism of~$\A^2$.
\end{enumerate}
\end{thm}

\begin{proof}
$(1)$ Let $P\in\C[x]$ be a non-zero polynomial such that $\Gamma$ is isomorphic to $\A^1\setminus\{x\in\A^1\mid P(x)=0\}$. Let $g\in\Aut(\Gamma)$ be an automorphism that fixes at most one point of~$\Gamma$. Let us denote also by $g$ its extension as an automorphism of~$\p^1$. We can assume that   $g$ fixes the point of~$\p^1$ at infinity, so that it is of the form  $x\mapsto ax+b$, for some $a\in\C^*$ and $b\in \C$. Moreover $P(ax+b)=\mu P(x)$ for some $\mu \in \C^{*}$.

When we embed $\Gamma$ into $\A^2$ via the map $x\mapsto (x,\frac{1}{P(x)})$, the automorphism $g$ extends to $(x,y)\mapsto (ax+b,\mu^{-1}y)$.

$(2)$ Let $g\in\Aut(\Gamma)$ be of finite order $n>1$ with $n$ odd such that it fixes $2$ points of~$\Gamma$. Suppose, for contradiction, that there exists a closed embedding $\tau\colon \Gamma\to \A^2$ for which $g$ extends  to an automorphism $h$ of~$\A^2$. Since $g$ has finite order $n$, the automorphism $h^n\in \Aut(\A^2)$ fixes pointwise the curve $\tau(\Gamma)$. Because $g$ fixes two points of~$\Gamma$, $\tau(\Gamma)$ is  not isomorphic to $\A^1$, hence $h^n$ is trivial by Lemma~\ref{Lemma:fixed points of a plane polynomial automorphism}.

Recall that every automorphism of~$\A^2$ of finite order is conjugate to a linear one (Lemma~\ref{Lemm:FiniteGL2}). Thus, there exists $f\in\Aut(\A^2)$  such that $\hat{h}=f\circ h\circ f^{-1}$ is linear. Moreover, the automorphism $g\in\Aut(\Gamma)$ extends to $\hat{h}$, when we consider the embedding $\hat{\tau}=f\circ\tau\colon \Gamma\to\A^2$.

The linear automorphism $\hat{h}$ extends to an automorphism of~$\p^2$, and the closure of~$\hat{\tau}(\Gamma)$ in $\p^2$ is a projective rational curve $C$, having all its singular points on the line at infinity $L=\p^2\setminus \A^2$.

If $C$ is smooth, it is isomorphic to $\p^1$. Hence, it is a conic or a line, and thus intersects $L$ into $1$ or $2$ points, which contradicts the fact that $g$ acts on $C$ with order~$n>2$ and with no fixed point at infinity. This implies that $C$ is singular.

Denote by $\eta_1\colon X_1\to \p^2$ the blow-up of the points of~$\p^2$ that are singular points of~$C$, and write $C_1\subset X_1$ the strict transform of~$C$ in $X_1$. If $C_1$ is singular, we denote by $\eta_2\colon X_2\to X_1$ the blow-up of the points of~$X_1$ that are singular points of~$C_1$, and write $C_2$ the strict transform of~$C_1$ in $X_2$. We continue like this until we end with a smooth curve $C_m\subset X_m$ such that the intersection of~$C_m$ with all curves contracted by $\eta_1\eta_2\dots\eta_m$ is transversal. Note that $C_m$ is isomorphic to $\p^1$. For $i=1,\dotsc,m$, the lift of~$\hat{h}$ yields an automorphism $h_i$ of~$X_i$ which preserves the curve $C_i$. It also  preserves the pull-back of~$\A^2$ in $X_i$, which is again isomorphic to $\A^2$.

For $i=1,\ldots,m$, we denote by $\mathcal{B}_i\subset C_i$ the (finite) set of points not lying in $\A^2$. Each point $p\in\mathcal{B}_i$ has a multiplicity $m(p)$ as a point of~$C_i$. This multiplicity is a positive integer and it is equal to $1$ if and only if $C_i$ is smooth at this point $p$. Denote by $\mathcal{B}_0$ the set of points of~$C_0=C\subset X_0=\p^2$ not lying in $\A^2$ and let us  use the same notation  as above for the multiplicities of the points of~$\mathcal{B}_0$.

Writing $d$ the degree of~$C\subset \p^2$, the geometric genus of~$C$ can be computed with  the following classical formula. (Note that it is equal to $0$, since $C$ is rational.)

\begin{equation}0=\frac{(d-1)(d-2)}{2}-\sum_{i=0}^m \sum_{p\in \mathcal{B}_i} \frac{m(p)\cdot (m(p)-1)}{2}.\tag{$\star$}\label{genus}\end{equation}

Let us now  prove the following assertion by descending induction on $j\le m$.

\begin{equation}
\begin{array}{l}
\text{Let } j\in \{1,\dots,m\} \text{ and let  }J\subset \mathcal{B}_j\text{ be an orbit under the action of~$h_j$.}\\
\text{Then } m(p)=m(p') \text{ for all } p,p'\in J, \text{ and the integer }\sum\limits_{p\in J} m(p) \text{ is a multiple of n.}\tag{$\diamond$}\label{diamond}\end{array}\end{equation}
For $j=m$, the assertion $(\diamond)$ holds for all orbits $J\subset \mathcal{B}_m$. Indeed,  $C_m$ is isomorphic to $\p^1$ and the action of~$h_m$ on $\mathcal{B}_m\subset C_m$ is fixed-point-free, so all orbits have size $n$ and all multiplicities are equal to $1$.

Then, we can prove $(\diamond)$ for $j<m$, assuming it holds for every integer $k$ with $j+1\leq k\leq m$. For this, let $J\subset \mathcal{B}_j$ be an orbit under the action of~$h_j$ and let us denote by $m_J$ the multiplicity $m(p)$ of a point $p\in J$. Note that this multiplicity does not depend of the choice of~$p$, since $h_j$ acts transitively on $J$.

If $m_J=1$, all points of~$J$ are smooth, and so the pull-back by $\eta_{j+1}$ of~$J$ consists of~$|J|$ points of multiplicity $m_j=1$. This implies $\sum_{p\in J} m(p)\in n\mathbb{N}$, by induction hypothesis.

If $m_J>1$, then all points of~$J$ are singular points of the curve $C_j$ and are thus blown-up by $\eta_{j+1}\colon X_{j+1}\to X_j$. The number $m_J$ is the multiplicity of the curve $C_j$ at the point $p\in J$. Denoting by $E_p\subset X_{j+1}$ the curve contracted by $\eta_{j+1}$ onto $p$, the number $m_J$ is the intersection number $E_p\cdot C_{j+1}$. This latter is equal to the sum of~$m_q(E_p)\cdot  m_q(C_{j+1})$, where $q$ runs through all points infinitely near to $p$, and where $m_q(E_p)$ and $m_q(C_{j+1})$ are the multiplicities of the strict transforms of~$E_p$ and $C_{j+1}$ at $q$, respectively. Note that $m_q(E_p)$ is equal to $0$ or $1$.

Therefore, the sum $\sum _{p\in J} m_J$ is equal to a sum of multiplicities of orbits in $\mathcal{B}_k$ for $k\ge j+1$. By induction hypothesis, it is a multiple of~$n$. This achieves to prove $(\diamond)$.

In order to finish the proof, we will show  how  Equation~(\ref{genus}) and Assertion~(\ref{diamond}) imply that the integers $\frac{(d-1)(d-2)}{2}$ and $d$ are both multiple of~$n$. Since the greatest common divisor of~$d$ and $\frac{(d-1)(d-2)}{2}$ is $1$ or $2$, this will contradict the assumption $n> 2$.

To show that $n$ divides $\frac{(d-1)(d-2)}{2}$, we decompose the sum of (\ref{genus}) according to orbits
$$\frac{(d-1)(d-2)}{2}=\sum_{j=0}^m \sum_{J\subset \mathcal{B}_j}\sum_{p\in J} \frac{m(p)\cdot (m(p)-1)}{2}.$$
By Assertion~(\ref{diamond}), the multiplicities $m(p)$ are all equal among the same orbit $J$, so $\sum_{p\in J} m(p)\cdot (m(p)-1)$ is a multiple of~$\sum_{p\in J} m(p)$, which is a multiple of~$n$ by (\ref{diamond}). Since $n$ is odd, $\sum_{p\in J}\frac{m(p)\cdot (m(p)-1)}{2}$ is also a multiple of~$n$, and so is $\frac{(d-1)(d-2)}{2}$.

It remains to show that $d$ is also a multiple of~$n$. For this, we observe that the intersection number $d=L\cdot C$ is the sum of multiplicities of all points of~$C$ that belong to $L$, as proper or infinitely near points. Since $L$ is invariant under the extension of the affine automorphism $\hat{h}$, the union of these points decomposes into orbits of~$h_j$ and the sum is then a multiple of~$n$ by Assertion~(\ref{diamond}).
\end{proof}

\begin{corollary}\label{Coro:NoExt}
There exist rational smooth affine curves $\Gamma$ with $\Aut(\Gamma) \not=1$ and such that for every closed embedding of~$\Gamma$ in $\A^2$, the identity on $\Gamma$ is its only automorphism which extends to an automorphism of~$\A^2$.
\end{corollary}

\begin{proof}
Let $\omega=e^{2i \pi/3}$ and  $a_1=1$. Let $a_2,\dots,a_k$ be complex numbers algebraically independent over $\Q$. We consider the curve $\Gamma=\p^1\setminus\Lambda$, where $\Lambda$ is the following set of~$3k$ points
$$\Lambda=\left\{[a_i\omega^j:1]\ |\ i=1,\dots,k,\ j=0,\dots,2\right\}.$$
The  map $h\colon [x:y]\mapsto [x:\omega y]$ is obviously an automorphism of~$\Gamma$. We will now prove that it generates the whole automorphism group $\Aut(\Gamma)$ if $k\ge 3$. This will conclude the proof, since $h$ and $h^2$ do not extend to automorphisms of~$\A^2$ by Theorem~\ref{Thm:PtsInfinity}.

Let $g\in \Aut(\Gamma)$ be an automorphism of~$\Gamma$. It extends to an automorphism of~$\p^1$ that preserves the set $\Lambda$. Let us denote this latter also by $g$.

Consider the $4$-tuple  $V=\left([1:1],[\omega:1],[\omega^2:1],[a_2:1]\right)$.  Since $a_2,\dots,a_k$ are algebraically independent over $\Q$, the image of~$V$ by $g$ is a $4$-tuple of points contained in the set $$S=\left\{[1:1],[\omega:1],[\omega^2:1],[a_2:1],[a_2\omega:1],[a_2\omega^2:1]\right\}.$$ Indeed, the cross-ratio of~$g(V)$ must be equal to the cross-ratio of~$V$, i.e.\ to $\omega(\omega-a_2)/(a_2-1)$.

The same argument with the $4$-tuple $\left([1:1],[\omega:1],[\omega^2:1],[a_3:1]\right)$ allows us to conclude that $g$ preserves the set $\left\{[1:1],[\omega:1],[\omega^2:1]\right\}$. Therefore, $g$ is either a power of~$h$, or it is one of the maps  $\varphi_i\colon [x:y]\mapsto [y:x\omega^i]$ with $i=0\ldots 2$.

Finally, note that $g$ cannot be one of the $\varphi_i$'s, since $\varphi_i$ sends the point $[a_2:1]$ onto the point $[\dfrac{1}{a_2\omega^i}:1]$, which does not belong to the set $S$.
\end{proof}

\begin{remark}
The proof of Corollary~\ref{Coro:NoExt} shows that if $k\geq3$ and if the set $\Lambda\subset\p^1$ is general among all sets of distinct  $3k$ points  invariant by the map $[x:y]\mapsto[x:\omega y]$, then for all closed embeddings of the  curve $\Gamma=\p^1\setminus\Lambda$ into $\A^2$, the identity is the only automorphism of~$\Gamma$ that extends to an automorphism of~$\A^2$.

On the contrary, when $k\leq2$,  every such curve $\Gamma$ admits an automorphism of order~2 and Proposition~\ref{Prop:ExtOrder2} below implies then that this latter extends to an automorphism of~$\A^2$ for a well-chosen closed embedding of~$\Gamma$ into $\A^2$.\end{remark}

\begin{proposition}\label{Prop:ExtOrder2}
Let $\Gamma$ be a rational smooth affine  curve and let $\sigma\in \Aut(\Gamma)$ be an automorphism of~$\Gamma$ of order $2$. There exists a closed embedding of~$\Gamma$ in $\A^2$ and an automorphism $\hat{\sigma}\in \Aut(\A^2)$ of order $2$ whose restriction to $\Gamma$ yields $\sigma$.
\end{proposition}

\begin{proof}
Let $\Gamma=\p^1\setminus \Lambda$, where $\Lambda$ is a finite set of points. Let us denote by $\sigma$ the extension of the automorphism $\sigma\in\Aut(\Gamma)$ as an automorphism of~$\p^1$. If it fixes at most one point of~$\Lambda$,  the result follows from Theorem~\ref{Thm:PtsInfinity}.

We can thus assume that the two points of~$\p^1$ fixed by (the extension of) $\sigma$ belong to $\Gamma$. Let  $p$ be a point of~$\Lambda$. Its  orbit $\{p,\sigma(p)\}$ is then contained in $\Lambda$. Let $C$ be the curve $C=\p^1\setminus\{p,\sigma(p)\}$. Note that $C$ is isomorphic to $\A^1\setminus\{0\}$ and that $\sigma$ restricts to an automorphism of~$C$. Remark that all automorphisms of~$\A^1\setminus\{0\}$ of order $2$ with two fixed points are conjugate to the automorphism $x\mapsto x^{-1}\in\Aut(\Spec(\C[x,x^{-1}]))$.  Therefore, there is a closed embedding $\tau\colon C\to\A^2$ whose image is the curve defined by the equation $$y^2-1=x^2$$
and such that $\sigma\in \Aut(C)$
extends to the automorphism $\hat\sigma\colon(x,y)\mapsto(-x,y)$. Moreover, the curve $\tau(\Gamma)$ is then  equal to a set of points of~$\tau(C)$ satisfying that $\prod_{i=1}^n (y-a_i)\not=0$, for some $n\ge 0$ and distinct $a_1,\dots,a_n\in \C\setminus \{\pm 1\}$.

Let $Y\subset \A^2$ be the closed curve defined by the equation
$$y^2-1=x^2\cdot  \left(\prod_{i=1}^n (y-a_i)\right)^2.$$

Consider finally the birational transformation of~$\A^2$ defined by

$$(x,y)\dasharrow \left(\frac{x}{\prod_{i=1}^n (y-a_i)},y\right),$$
which restricts to an isomorphism between $\tau(\Gamma)$ and $Y$. Since it commutes with the automorphism $(x,y)\mapsto (-x,y)$, this yields the result.
\end{proof}
%\begin{remark}
%We finish this section by explaining why Theorem~\ref{Thm:PtsInfinity}(1) and Proposition~\ref{Prop:ExtOrder2} do not work for \emph{all} closed embeddings.
%This is a direct application of Proposition~\ref{Prop:A101}!
%
%
%Alternatively, this is shown on the curve $\Gamma\subset\A^2$ given by
%\[x^2(x-1)y=1.\]
%This curve is isomorphic to $\A^1\setminus\{0,1\}$ and has thus a group of automorphism isomorphic to $\Sym_3$ but no non-trivial element of~$\Sym_3$ extends to an automorphism of~$\A^2$, even if all involutions in $\Aut(\Gamma)$ have exactly one fixed point on $\Gamma$. It remains to see that any  automorphism $f\colon (x,y)\mapsto (f_1(x,y),f_2(x,y))$ of~$\A^2$ that preserves $\Gamma$ is in fact trivial.
%
%The action $f^*$ action on $\C[x,y]$ sends the generator of the ideal of~$\Gamma$ on another one, which implies that
%$$f_1^2(f_1-1)f_2-1=\lambda(x^2(x-1)y-1)$$ for some $\lambda\in \C^*$. This implies that $\deg(f_1)=\deg(f_2)=1$, which corresponds to saying that $f$ is affine, i.e\;that it extends to an automorphism $\tau$ of~$\p^2$, that preserves the curve of equation
%$$x^2(x-z)y-z^4=0.$$
%The point $[0:1:0]$ is the only singular point and is thus fixed, and the lines $x=0$, $x=z$ are invariant. Since $z=0$ is also invariant, the action is of the form $[x:y:z]\mapsto [x:\mu y:z]$ and the form for the equation implies that $\mu=1$.
%\end{remark}

\section{Planar embeddings in the space}\label{SecPlanarEmbedd}

The following question of Bhatwadekar and Srinivas is asked at the end of \cite{Srinivas}: are any two embeddings of a smooth affine curve in $\A^2$  equivalent, when considered as embeddings in $\A^3$?

The next result answers positively for the case of rational smooth affine curves.

\begin{proposition}\label{Prop:planarEmb}
Let $\Gamma$ be a rational smooth affine  curve.
\begin{enumerate}
\item
If $\tau_1,\tau_2\colon \Gamma\to \A^3$ are two closed embeddings whose images are contained in a hyperplane $($planar embeddings in the space$)$, there exists an automorphism $\alpha\in \Aut(\A^3)$ such that $\tau_2=\alpha\circ \tau_1$, i.e.\ any two planar embeddings in the space are equivalent.
\item
In particular, fixing a planar embedding $\Gamma\to \A^3$, every automorphism of~$\Gamma$ extends to $\A^3$.
\end{enumerate}
\end{proposition}

\begin{proof}
Let $\Gamma=\A^1\setminus\{x\in\A^1\mid P(x)=0\}$, where $P \in \C[x]$ is a polynomial with simple roots.  Note that the coordinate ring of~$\Gamma$ is $\C[\Gamma]=\C[x,\frac{1}{P(x)}]$ and recall that the map $x\mapsto(x,\frac{1}{P(x)})$ defines a closed embedding of~$\Gamma$ in $\A^2$. To prove the proposition, it suffices to prove that any planar embedding  is equivalent to the one  given by $x\mapsto (x,\frac{1}{P(x)},0)$.

Let $\tau\colon \Gamma \to \A^3$ be a planar embedding of~$\Gamma$. We can compose $\tau$ with an affine automorphism $f_1$ of~$\A^3$ and get an embedding $\tau_2=f_1\circ\tau\colon \Gamma \to \A^3$ of the form $x\mapsto \left(0,  Q(x),R(x)\right)$, where $Q,R\in \C(x)$ are rational functions without poles on $\Gamma$.
Since $\tau_2$ is a closed embedding of the curve $\Gamma$, the equality  $\C[x,\frac{1}{P(x)}]=\C[Q(x),R(x)]$ holds.
In particular, there exists a polynomial $A\in \C[X,Y]$ such that $A(Q(x),R(x))=x$. Now, we compose $\tau_2$ with the automorphism of~$\A^3$ defined by $f_2(X,Y,Z)=(X+A(Y,Z),Y,Z)$ and obtain the embedding $\tau_3\colon \Gamma\to \A^3$ given by
$$\tau_3 \colon x \mapsto\left(x,Q(x),R(x)\right).$$

Because of the equality $\C[x,\frac{1}{P(x)}]=\C[Q(x),R(x)]$, all zeros of~$P(x)$ are poles of~$aQ(x)+bR(x)$ for general complex numbers $a,b\in \C$. We can thus compose $\tau_3$ with a linear automorphism of the form $(X,Y,Z)\mapsto(X,aY+bZ,Z)$ and get an embedding $\tau_4\colon\Gamma\to\A^3$ of the form
$$\tau_4 \colon x \mapsto\left(x,  \frac{Q_1(x)}{Q_2(x)},\frac{R_1(x)}{R_2(x)}\right),$$
\noindent where $Q_1,Q_2,R_1,R_2\in\C[x]$ are polynomials such that $Q_1$ and $Q_2$ (resp.\ $R_1$ and $R_2$) have no common factor, and such that $P(x)$ divides $Q_2(x)$.

In particular, there exist two polynomials $U,V\in \C[x]$ such that $UQ_1+VP=1$. It follows
$$\frac{1}{P}=\frac{UQ_1+VP}{P}=U\frac{Q_1}{P}+V=SU\frac{Q_1}{Q_2}+V,$$
where $S\in\C[x]$ satisfies $PS=Q_2$.

This implies $\C[x,\frac{1}{P}]\subset\C[x,\frac{Q_1}{Q_2}]$ and thus $$\C[x,\frac{Q_1}{Q_2},\frac{R_1}{R_2}]=\C[x,\frac{1}{P}]=\C[x,\frac{Q_1}{Q_2}].$$

Therefore, there exist polynomials $B,C\in \C[X,Y]$ such that $B(x,\frac{Q_1(x)}{Q_2(x)})=\frac{1}{P(x)}-\frac{R_1(x)}{R_2(x)}$ and $C(x,\frac{1}{P(x)})=\frac{Q_1(x)}{Q_2(x)}$. Finally, we consider the automorphisms of~$\A^3$ defined by $f_4(X,Y,Z)=(X,Y,Z+B(X,Y))$ and $f_5(X,Y,Z)=(X,Z,Y-C(X,Z))$. One checks that $f_5\circ f_4\circ\tau_4\colon\Gamma\to\A^3$ is  the desired embedding $x\mapsto (x,\frac{1}{P(x)},0)$.
\end{proof}

Note that the proof above is constructive. In particular, a planar embedding of a smooth rational curve $\Gamma$ in $\A^3$ and an automorphism $\varphi$ of~$\Gamma$ being given, it allows us to construct an explicit automorphism of~$\A^3$ which extends $\varphi$.

\begin{example}
Let $\Gamma$ be the curve $\Gamma=\A^1\setminus\{0,1\}$ and let $\rho\in\Aut(\Gamma)$ be the automorphism of~$\Gamma$ defined by $\rho(x)=1/(1-x)$. We saw in Section~\ref{Sec:planeEmb} that there is no closed embedding of~$\Gamma$ into $\A^2$ such that $\rho$ extends to an automorphism of~$\A^2$. However, it extends to an automorphism of~$\A^3$, when we consider the embedding  $\tau\colon \Gamma\to\A^3$ defined by $x\mapsto(x,1/x(x-1),0)$.

Following the proof of Proposition~\ref{Prop:planarEmb}, we let $f_1,f_2,\ldots,f_5$ be the automorphisms of~$\A^3$ defined by $f_1(X,Y,Z)=(Z,Y,X)$, $f_2(X,Y,Z)=(X+Y+2-YZ^2,Y,Z)$, $f_3(X,Y,Z)=(X,aY+bZ,Z)$, $f_4(X,Y,Z)=(X,Y,Z-\frac{1}{ab}[(b+(a-b)X)(Y-aX+2a)-(a-b)^2](1+X))$ and $f_5(X,Y,Z)=(X,Z,Y-aX+2a+aZ+(b-a)XZ)$, where $a,b\in\C$ are general complex numbers.

Setting  $F=f_5\circ f_4\circ\cdots\circ  f_1$, one checks $F\circ\tau\circ\rho=\tau$. This implies that $F^{-1}$ is an extension of the automorphism $\rho\in\Aut(\Gamma)$.
\end{example}

\begin{remark}
To our knowledge, there is no known example of a smooth affine curve admitting two non-equivalent embeddings into $\A^3$.
Paradoxically, we do not know any smooth affine curve such that all its embeddings into $\A^3$ are equivalent!

The case of the affine line is of particular interest. On one hand, all closed embeddings of~$\A^1$ into $\A^2$ are equivalent by the famous Abhyankar-Moh-Suzuki theorem. On the other hand, all closed embeddings of~$\A^1$ into $\A^n$ with $n\geq4$ are also equivalent (see \cite{Srinivas} or \cite{Kaliman}).
\end{remark}

\section{Actions of~$\SL(2,\C)$ on $\End(\A^2)$ and of~$\PGL(2,\C)$ on $\p^1$}\label{SecSL2}

The aim of this section is to construct, for every non-empty subset $\Lambda$ of~$\p^1$ that is invariant by a subgroup $H$ of~$\Aut(\p^1)$, a $H$-equivariant endomorphism of~$\p^1$ whose fixed-point set is equal to the set $\Lambda$ (Corollary~\ref{Coro:ExistHequi}). We will use this result later on to construct embeddings of every rational smooth  affine curve into $\A^3$ in such a way that the whole automorphism group of the curve extends to a subgroup of~$\Aut(\A^3)$.

\medskip

For the rest of the paper we will consider the following actions of the group $\SL(2,\C)$ on $\mathcal{O}(\A^2)=\C[x,y]$ and $\End(\A^2)=\C[x,y]\times \C[x,y]$.
$$\begin{array}{ccc}
\SL(2,\C) \times \mathcal{O}(\A^2)& \to & \mathcal{O}(\A^2)\\
(g,P)& \mapsto & g\cdot P: =P\circ g^{-1}\end{array}$$
and
$$\begin{array}{ccc}
\SL(2,\C) \times \End(\A^2)& \to & \End(\A^2)\\
(g,F)& \mapsto & g\cdot F:= g \circ F \circ g^{-1}.\end{array}$$

Note that these actions come from the natural action of~$\SL(2,\C)$ on $\A^2$. Indeed, denote by $V$ the space $\A^2$ as a complex vector space of dimension~$2$ and identify the set of the linear forms on it as the dual space $V^{*}$. The action of~$\SL(V)$ on $V$ yields   actions on $V^{*}$ , on the symmetric algebra $S(V^{*})$ and on $S(V^{*})\otimes V$. The natural isomorphisms between $S(V^{*})$ and   $\C[x,y]=\mathcal{O}(\A^2)$ , and between $S(V^{*})\otimes V$ and $\C[x,y]\times \C[x,y]=\End(\A^2)$, lead then to the $\SL(2,\C)$-actions that we  defined above.

\begin{lemma}\label{Lemm:SL2Cequivariant}
The map $\rho\colon \End(\A^2)\to \mathcal{O}(\A^2)$ defined by
$$\begin{array}{ccc}
\C[x,y]\times \C[x,y]& \to & \C[x,y]\\
(f_1,f_2)&\mapsto &f_1y-f_2x \end{array}
$$
is $\SL(2,\C)$-equivariant, when we consider the actions defined above.
\end{lemma}
\begin{proof}
The result could of course be checked by direct computations, but let us  mention that it also follows from the fact that $\rho$ corresponds to the morphism $S(V^{*})\otimes V\to S(V^{*})$ given by the composition $\tau_2\circ \tau_1$, where $\tau_1$ and $\tau_2$ are the two following homomorphisms of~$\SL(V)$-modules.
$$\begin{array}{rrcl}
\tau_1\colon &S(V^*) \otimes V &\to& S(V^*) \otimes V \otimes V^* \otimes V \\
&p \otimes v& \mapsto &p \otimes v \otimes {\rm id},\\
\end{array}$$
where ${\rm id}$ denotes the identity element seen as an element of~$ V^* \otimes V = {\rm Hom}(V,V)$, and

$$\begin{array}{rrcl}
\tau_2\colon & S(V^*) \otimes V \otimes V^* \otimes V  &\to& S(V^*)\\
&p \otimes v_1 \otimes v_2^* \otimes v_3 &\mapsto& \det (v_1,v_3) (p v_2^*).\\
\end{array}$$
\end{proof}

\begin{lemma}\label{Lemm:EpFixed}
Let $G\subset \SL(2,\C)$ be a finite subgroup of~$\SL(2,\C)$ and let $P\in\C[x,y]$. The following conditions are equivalent:

\begin{enumerate}
\item
The polynomial $P$ satisfies $P(0,0)=0$ and is fixed by $G$.
\item
There exists an endomorphism  $F=(f_1,f_2)$ of~$\A^2$ that is fixed by $G$ and  such that $\rho(F)=f_1y-f_2x=P$.
\end{enumerate}
\end{lemma}

\begin{proof}
Let $E_P\subset \End(\A^2)$ be the set
$$E_P=\rho^{-1}(P)=\left\{(f_1,f_2)\in \C[x,y]\times \C[x,y]\ |\ f_1y-f_2x=P\right\}.$$
This defines an affine subset of the $\C$-vector space $\End(\A^2)$, since the endomorphism $(\lambda f_1+(1-\lambda)f_3,\lambda f_2+(1-\lambda)f_4)$ belongs to $E_P$, for any $(f_1,f_2), (f_3,f_4)\in E_P$ and any $\lambda\in \C$. Moreover, $E_P$ is non-empty if and only if $P(0,0)=0$.

If $F\in\End(\A^2)$ is fixed by $G$ and belongs to $E_P$, then
$$g\cdot P=g\cdot\rho(F)=\rho(g\cdot F)=\rho(F)=P$$
hold for any $g\in G$. This shows $(2)\Rightarrow (1)$.

If $P$ is fixed by $G$, then the set $E_P$ is invariant by $G$, since
$$\rho(g\cdot F)=g\cdot\rho(F)=g\cdot P=P$$
hold for any $F\in E_P$ and $g\in G$.

Therefore, if $F$ belongs to $E_P$, then $\frac{1}{|G|} \sum_{g\in G} g\cdot F$ is an element of~$E_P$ that is fixed by $G$. This shows $(1)\Rightarrow (2)$ and concludes the proof.
\end{proof}

\begin{proposition}\label{Prop:TechGH}
Let $H\subset \PGL(2,\C)=\Aut(\p^1)$ be a finite subgroup and set $G=\pi^{-1}(H)$, where $\pi\colon \SL(2,\C)\to \PGL(2,\C)$ is the canonical surjective map. Let $\Lambda\subset \p^1$ be a non-empty $H$-invariant finite subset.
\begin{enumerate}
\item
There exist homogeneous polynomials $f_1,f_2\in \C[x,y]$ of the same degree such that
$(f_1,f_2)$ is an endomorphism of~$\A^2$ fixed by $G$ and such that
$$\Lambda=\left\{[x:y]\in \p^1\mid f_1(x,y)y-f_2(x,y)x=0\right\}.$$
\item
The morphism $\delta\colon \p^1 \to \p^1$ defined by $$[x:y]\mapsto [f_1(x,y):f_2(x,y)]$$ is $H$-equivariant, for all pairs $(f_1,f_2)$ given by the statement $(1)$ above.
\item
There exist polynomials $f_1,f_2$ satisfying the statement $(1)$ and also the extra property  $$\Lambda=\left\{q\in \p^1 \mid \delta(q)=q\right\}.$$
This latter holds moreover for all pairs $(f_1,f_2)$ given by the statement $(1)$, in the case where the set $\Lambda$ consists of  exactly one orbit of~$H$.
\end{enumerate}
\end{proposition}

\begin{proof}
$(1)$ We let $p\in \C[x,y]$ be the (unique up a nonzero constant) square-free homogeneous polynomial whose roots correspond to the points of~$\Lambda$. Because $\Lambda$ is invariant by $H$, there exists a character $\chi \colon G\to \C^{*}$ such that
$$p\circ g=\chi(g)p,$$
for all $g\in G$. Since $G$ is finite, there exists a positive integer $d$ such that the polynomial $P=p^d$ is fixed by $G$.

By Lemma~\ref{Lemm:EpFixed}, there exists an endomorphism $(f_1,f_2)\in \C[x,y]\times \C[x,y]$ of~$\A^2$ that is fixed by $G$ and such that $f_1y-f_2x=P$. Since $P$ is homogeneous and since the action of~$G$ on $\End(\A^2)$ is linear and preserves the filtration by degrees, we can assume that $f_1$ and $f_2$ are homogeneous of the same degree. This proves $(1)$.

Statement $(2)$ follows directly from the fact that the endomorphism $(f_1,f_2)$ is fixed by $G$.

$(3)$ Since $\delta$ is $H$-equivariant, its fixed-point set is invariant by $H$. Let us denote it by $\Omega_{\delta}$ and write $f_1=\alpha\tilde{f_1}$ and $f_2=\alpha\tilde{f_2}$, where $\alpha, \tilde{f_1}, \tilde{f_2}$ are homogeneous polynomials such that $\tilde{f_1}$ and $\tilde{f_2}$ have no common root in $\p^1$. Then,  $\delta([x:y])=[\tilde{f_1}(x,y):\tilde{f_2}(x,y)]$ holds for all  $[x:y]\in \p^1$. The set $\Omega_\delta=\{q\in \p^1 \ |\ \delta(q)=q\}$ is thus the zero set of~$\tilde{f_1}y-\tilde{f_2}x$. In particular, it is non-empty. Moreover, the equalities $P=f_1 y - f_2 x =\alpha(\tilde{f_1}y-\tilde{f_2}x)$ imply that $\Omega_{\delta}$ is contained in $\Lambda$.

If $\Lambda$ consists of exactly one orbit of~$H$, then $\Omega_\delta=\Lambda$ follows from the fact that $\Omega_{\delta}$ is invariant by $H$.

Let us now consider the general case, where $\Lambda$ consists of~$r>1$ orbits of~$H$ and write $\Lambda=\bigcup_{i=1}^r \Lambda_i$, where $\Lambda_1,\dots,\Lambda_r$ are disjoint orbits of~$H$. For each $i$, there exist,  by the previous argument, homogeneous polynomials $f_{i,1}, f_{i,2}$ of the same degree such that the zero set of~$P_i=f_{i,1}y-f_{i,2}x$ is equal to $\Lambda_i$ and such that the pair $(f_{i,1} ,f_{i,2})$ defines an endomorphism of~$\A^2$ which is fixed by $G$.

Set
$$g_1=\frac{1}{r}\sum_{i=1}^r \left(f_{i,1} \prod_{j\not=i} P_j\right)\quad\text{and}\quad
g_2=\frac{1}{r}\sum_{i=1}^r \left(f_{i,2} \prod_{j\not=i} P_j\right).$$
Note that $g_1$ and $g_2$ are homogeneous of the same degree and satisfy the equality $g_1y-g_2x=\prod_{i=1}^r P_i$. Moreover,  the endomorphism $(g_1,g_2)\in \End(\A^2)$ is fixed by $G$. In other words, it satisfies the statement $(1)$ of the lemma.

We will now show that the set $\Omega_{\tilde{\delta}}$ of fixed points of the morphism $\tilde{\delta}\colon\p^1\to\p^1$ defined by $\tilde{\delta}([x:y])=[g_1(x,y):g_2(x,y)]$ is  equal to $\Lambda$, which will conclude the proof. Note that it is contained in $\Lambda$ and invariant under the action of  $H$, since $\tilde{\delta}$ is $H$-equivariant.

Let us write $g_1=\beta \tilde{g_1}$ and $g_2=\beta\tilde{g_2}$, where $\beta, \tilde{g_1},\tilde{g_2}$ are homogeneous and $\tilde{g_1},\tilde{g_2}$ have no common root in $\p^1$. Note that the set $\Omega_{\tilde{\delta}}$ is equal to the zero set of the homogeneous polynomial $\tilde{g_1}y-\tilde{g_2}x$ .

We claim that none of the $P_i$ divides $\beta$. Indeed, otherwise such a $P_i$ would divide both $g_1$ and $g_2$ and thus also $f_{i,1} \prod_{j\not=i} P_j$ and  $f_{i,2} \prod_{j\not=i} P_j$. Since $P_i$ has no common root with any of the $P_j$, this would imply that $P_i$ divides $f_{i,1}$ and $f_{i,2}$.  This is impossible, since  $P_i=f_{i,1}y-f_{i,2}x$, hence $P_i$ has degree bigger than $f_{i,1}$ and $f_{i,2}$.

Therefore, it follows  from the equalities $$\prod_{i=1}^r P_i=g_1y-g_2x=\beta( \tilde{g_1}y-\tilde{g_2}x)$$
that, for every index $i$, at least one point of~$\Lambda_i$ is contained in $\Omega_{\tilde{\delta}}$. This latter set being invariant by $H$ and  $\Lambda_i$ being an orbit under the action of~$H$, we get that the whole set $\Lambda_i$ is contained in $\Omega_{\tilde{\delta}}$, for each $i=1\ldots r$. This achieves the proof.
\end{proof}

\begin{corollary}\label{Coro:ExistHequi}
Let $H\subset \PGL(2,\C)=\Aut(\p^1)$ be a finite subgroup and let $\Lambda\subset \p^1$ be a finite subset. The following conditions are equivalent:

\begin{enumerate}
\item
The set $\Lambda$ is non-empty and invariant by $H$.
\item
There exists a $H$-equivariant morphism $\delta\colon \p^1\to \p^1$ such that $$\Lambda=\{q\in \p^1\ |\ \delta(q)=q\}.$$
\end{enumerate}
\end{corollary}

\begin{proof}
The implication $(1)\Rightarrow (2)$ follows directly from Proposition~\ref{Prop:TechGH}. Let us prove the other one.

Let $\delta\colon\p^1\to\p^1$ be a $H$-equivariant morphism whose fixed-point set is equal to $\Lambda$. The set $\Lambda$ is then invariant under the action of~$H$, since  $\delta(h(q))=h(\delta(q))=h(q)$ hold  for all $h\in H$ and all $q\in\Lambda$.

Furthermore, let $f_1,f_2 \in \C [x,y]$ be two homogeneous polynomials of the same degree and without common root in $\p^1$ such that  $\delta ( [x:y]) = [f_1(x,y) : f_2 (x,y) ]$ for all points $[x:y] \in \p^1$. Since $\Lambda$ is the zero set of~$f_1y-f_2x$, it is clearly non-empty.
\end{proof}

\section{Equivariant embeddings into the affine three-space}\label{Sec:EmbQuad}

Let us recall that the following morphism
$$\begin{array}{rcl}
\p^1\times \p^1&\hookrightarrow &\p^3\\
\left([y_0:y_1],[z_0:z_1]\right)& \mapsto &[y_0z_0:y_0z_1:y_1z_0:y_1z_1]\end{array}$$
is a classical closed embedding of~$\p^1\times\p^1$ into $\p^3$ and that it induces an isomorphism between  $\p^1\times \p^1$ and the quadric in $\p^3$ defined by the equation $x_0x_3=x_1x_2$. Moreover, since this embedding is canonical (it is given by the linear system  $\lvert-\frac{1}{2}K_{\p^1\times\p^1}\rvert$), every automorphism of~$\p^1\times \p^1$ extends to a unique automorphism of~$\p^3$.

Identifying $\A^3$ with the complement in $\p^3$ of the hyperplane defined by the equation $x_1=x_2$, we obtain a closed embedding $(\p^1\times \p^1)\setminus \Delta\hookrightarrow \A^3$,
where $\Delta$ denotes the diagonal curve $\Delta=\{(q,q)\ |\ q \in \p^1\}\subset \p^1\times \p^1$.

Consider the diagonal action of~$\PGL(2,\C)=\Aut(\p^1)$ on $\p^1\times\p^1$. Note that each automorphism of~$\p^1\times\p^1$ coming from this action extends to an automorphism of~$\p^3$ which preserves the plane of equation $x_1=x_2$. This yields an  action of~$\PGL(2,\C)$ on $\A^3$ for which the  closed embedding  $(\p^1\times \p^1)\setminus \Delta\hookrightarrow \A^3$, that we defined above,  becomes  $\PGL(2,\C)$-equivariant.

After a change of coordinates in $\A^3$, we obtain a $\PGL(2,\C)$-equivariant embedding of~$(\p^1\times\p^1)\setminus\Delta$ into $\A^3$, where the action of~$\PGL(2,\C)$ on $\A^3$ is linear.

\begin{lemma} \label{Lem:ClosedEmbP1P1dC3}
The morphism $$\begin{array}{cccl}
\iota\colon& (\p^1\times \p^1)\setminus \Delta&\hookrightarrow& \A^3\\
&([y_0:y_1],[z_0:z_1]) & \mapsto & \left(\dfrac{y_0z_1+y_1z_0}{y_0z_1-y_1z_0}, \dfrac{2y_0z_0}{y_0z_1-y_1z_0},\dfrac{2y_1z_1}{y_0z_1-y_1z_0}\right)\end{array}$$
is a  closed embedding whose image is  the hypersurface of~$\A^3$ defined by the equation $yz=x^2-1$.

Moreover, $\iota$ is $\PGL(2,\C)$-equivariant, when we consider the actions of~$\PGL(2,\C)$ on $ (\p^1\times \p^1)\setminus \Delta$ and  $\A^3$  defined by
\[\begin{array}{rcl}
\PGL(2,\C)\times (\p^1\times \p^1)\setminus \Delta& \to &(\p^1\times \p^1)\setminus \Delta\\
\left( 
\left(\begin{smallmatrix} a & b \\ c & d \end{smallmatrix} \right)
,([y_0:y_1],[z_0:z_1]) \right)&\mapsto& ([ay_0+by_1:cy_0 +dy_1],[az_0+bz_1:cz_0+dz_1]) \end{array}\]
and
$$\begin{array}{rcl}
\PGL(2,\C)\times\A^3& \to &\A^3 \\
\left(\left(\begin{array}{cc}
a & b \\
c & d\end{array}\right), \left(\begin{array}{c} x \\ y \\ z \end{array}\right) \right)&\mapsto&  \dfrac{1}{ad-bc}  \left(\begin{array}{ccc} ad+bc&ac&bd\\ 2ab&a^2&b^2\\ 2cd&c^2&d^2 \end{array}\right) \cdot \left(\begin{array}{c} x \\ y \\ z \end{array}\right).  \end{array}$$
\end{lemma}

\begin{proof}
Let $Q$ denotes the quadric hypersurface of~$\A^3$ defined by the equation $yz=x^2-1$. One checks that the morphism $\iota$ induces an isomorphism between $(\p^1\times \p^1)\setminus\Delta$ and $Q$ whose inverse morphism is given by
$$\begin{array}{rcl}
Q& \to &\ \ \  \ (\p^1\times \p^1)\setminus\Delta\\
(x,y,z)&\mapsto &\left\{\begin{array}{ll} ([x+1:z],[y:x+1])& \mbox{ if } x\not=-1, \\
([y:x-1],[x-1:z]) & \mbox{ if } x\not=1.\end{array}\right.
\end{array}$$
It is also straightforward to check  that $\iota$ is $\PGL(2,\C)$-equivariant for the given actions.
\end{proof}

Combining the latter lemma with the results of the previous section, we finally get $\Aut(\Gamma)$-equivariant embeddings of every smooth affine rational curve $\Gamma$ into $\A^3$.

\begin{thm}\label{thm:GroupXC3}
For every rational smooth affine  curve $\Gamma$, there exist a linear action of~$\Aut(\Gamma)$ on $\A^3$ and a closed embedding $\tau\colon \Gamma\hookrightarrow \A^3$ which is $\Aut(\Gamma)$-equivariant for this action.
\end{thm}

\begin{proof}
If $\Gamma=\A^1$, it suffices to consider the embedding $\tau\colon \A^1\to\A^3$ defined by $\tau(t)=(t,0,0)$, and to let $\Aut(\Gamma)=\left\{x\mapsto ax+b\mid a\in \C^{*}, b\in \C\right\}$ act on $\A^3$ via the maps $(x,y,z)\mapsto (ax+b(y+1),y,z)$.

If $\Gamma=\C^{*}$, we consider the embedding $\tau\colon \Gamma\to\A^3$ defined by $\tau(t)=(t,1/t,0)$. Its image is the curve in $\A^3$ defined by the equations $z=0$ and $xy=1$. Recall that $$\Aut(\Gamma)=\left\{ \varphi_{\lambda}\colon x \mapsto \lambda x \mid \lambda \in \C ^* \right\} \cup \left\{\psi_{\lambda}\colon x \mapsto \lambda x^{-1} \mid \lambda \in \C ^* \right\}.$$ The embedding $\tau$ becomes $\Aut(\Gamma)$-equivariant, when we let $\Aut(\Gamma)$ act on $\A^3$ via the maps $\Phi_{\lambda}\colon(x,y,z)\mapsto (\lambda x, \lambda^{-1} y,z)$ and $\Psi_{\lambda}\colon(x,y,z)\mapsto (\lambda y,\lambda^{-1} x,z)$.

If $\Gamma$ is equal to $\p^1\setminus \Lambda$, where $\Lambda$ is  a finite set  of at least $3$ points, then its automorphism group $H=\Aut(\Gamma)$ is the finite subgroup of~$\PGL(2,\C)=\Aut(\p^1)$ that preserves the set $\Lambda$.
Applying Corollary~\ref{Coro:ExistHequi}, let  $\delta\colon \p^1\to \p^1$ be a  $H$-equivariant morphism such that $\Lambda=\left\{q\in \p^1\mid\delta(q)=q\right\}$. This allows us to define a closed embedding
$\hat{\tau}\colon \Gamma\to (\p^1\times \p^1)\setminus \Delta$  by letting $\hat{\tau}(q)=(q, \delta(q))$ for all $q\in \Gamma=\p^1\setminus \Lambda$. The morphism~$\hat{\tau}$ is moreover $H$-equivariant, when $H$ acts diagonally on $(\p^1\times \p^1)\setminus \Delta$.

Composing $\hat{\tau}$ with the $\PGL(2,\C)$-equivariant closed embedding $\iota\colon (\p^1\times \p^1)\setminus \Delta\hookrightarrow \A^3$ that we defined in Lemma~\ref{Lem:ClosedEmbP1P1dC3}, we obtain a closed embedding $\tau\colon \Gamma\to \A^3$  which is $H$-equivariant, as desired.
\end{proof}

\section{Explicit formulas for the equivariant embeddings into $\A^3$}\label{Sec:Explicit}
The proof of Theorem~\ref{thm:GroupXC3} is constructive and already contains explicit $\Aut(\Gamma)$-equivariant embeddings into $\A^3$ for the curves $\Gamma=\A^1$ and $\Gamma=\A^1\setminus\{0\}$.  Let us now describe the construction for the other cases, i.e., when the automorphism group $\Aut(\Gamma)$ is finite.

We consider the curves $\Gamma=\p^1\setminus\Lambda$, where $\Lambda$ is a set of at least 3 points of~$\p^1$. Let us denote by $H$ the subgroup of~$\Aut(\p^1)=\PGL(2,\C)$ that restricts to $\Aut(\Gamma)$, and  denote as before by $G$ its pull-back in $\SL(2,\C)$, which is a finite group of order $2\lvert H\rvert$.  The set $\Lambda$ decomposes into $r$ orbits $\Lambda=\bigcup_{i=1}^r \Lambda_i$ of~$H$. An orbit $\Lambda_i$ of~$H$ is given by the zero set of a homogeneous polynomial $p_i\in \C[x,y]$.  Some power $P_i=p_i^{d_i}$ of~$p_i$ is invariant by the action of~$G$ on $\p^1$ defined in Section~\ref{SecSL2}.  For each $i$, Lemma~\ref{Lemm:EpFixed} yields the existence of
a $G$-invariant pair $(f_{i,1},f_{i,2})\in \End(\A^2)$ which satisfy $f_{i,1}y-f_{i,2}x=P_i$. The $H$-equivariant morphism $\delta\colon \p^1 \to \p^1$ given by Proposition~\ref{Prop:TechGH} (or Corollary~\ref{Coro:ExistHequi}) is thus $\delta\colon[x:y]\dasharrow [f_1(x,y):f_2(x,y)]$, where
$$f_1=\frac{1}{r}\left(\prod_{i=1}^r P_i\right)\sum_{i=1}^r \frac{f_{i,1}}{P_i}\quad\text{and}\quad
f_2=\frac{1}{r}\left(\prod_{i=1}^r P_i\right)\sum_{i=1}^r \frac{f_{i,2}}{P_i}.$$
Moreover, $(f_1,f_2)$ is invariant by $G$ and satisfies $f_1y-f_2x=\prod_{i=1}^r P_i$.

Following the proof of Theorem~\ref{thm:GroupXC3}, we define a closed  embedding $\Gamma = \p^1 \setminus \Lambda \to (\p^1\times \p^1) \setminus \Delta$ by $[x:y]\mapsto ([x:y], [f_1:f_2])$. We compose then this latter with the embedding $\iota\colon  (\p^1\times \p^1) \setminus \Delta\to \A^3$ defined by Lemma~\ref{Lem:ClosedEmbP1P1dC3}, and obtain the following $\Aut(\Gamma)=H$-equivariant closed embedding of~$\Gamma$ into $\A^3$.
$$\begin{array}{rcl}
\Gamma = \p^1 \setminus \Lambda &\to& \A^3\\
{ }[x:y]&\mapsto&\displaystyle{ \left(\frac{1}{r}\sum_{i=1}^r \frac{x f_{i,2}+yf_{i,1}}{xf_{i,2}-yf_{i,1}},\frac{1}{r}\sum_{i=1}^r \frac{2x f_{i,1}}{xf_{i,2}-yf_{i,1}},\frac{1}{r}\sum_{i=1}^r \frac{2y f_{i,2}}{xf_{i,2}-yf_{i,1}}\right)}.\end{array}$$
So it suffices to determine the polynomials $f_{i,1}$ and  $f_{i,2}$, which depend on $H$ and $\Lambda$, to get explicit embeddings.

\medskip

Recall that any finite subgroup of~$\Aut(\p^1)=\PGL(2,\C)$ is isomorphic to
$\mathbb{Z}/n\mathbb{Z}$ (the cyclic group of order $n$), $D_{2n}$ (the dihedral group of order $2n$),
${\mathfrak A}_4$ (the tetrahedral group), ${\mathfrak S}_4$
(the octahedral or cubic group)
or ${\mathfrak A}_5$ (the icosahedral or dodecahedral group)
and that there is only one conjugacy class for each of these groups
(see e.g.~\cite{Beauville}).

$1)$ In the cyclic case, we can assume that $H\subset \PGL(2,\C)$ is generated by $[x:y]\mapsto [\xi_n x:y]$, where $\xi_n$ is a primitive $n$-th root of unity. Its pullback $G\subset \SL(2,\C)$ is then generated by $\left( \begin{array}{cc} \zeta & 0\\ 0 & \zeta^{-1}  \end{array} \right)$, where $\zeta$ is a primitive $2n$-th root of unity.
An orbit $\Lambda_i$ of~$H$ is given by the zero set of a polynomial $p_i=a_i x^n+b_i y^n$ for some $(a_i,b_i)\in \C^2 \setminus \{(0,0)\}$ (the cases where $a_i=0$ or $b_i=0$ provide a fixed point with multiplicity $n$). We thus get
$$P_i=(p_i)^2\in \mathcal{O}(\A^2)^G$$
and
$$\left(f_{i,1},f_{i,2}\right)=\left(b_iy^{n-1}(a_i x^n+b_i y^{n}),-a_ix^{n-1}(a_i x^n+b_i y^{n})\right)\in \End(\A^2)^G$$ which satisfy $f_{i,1}y-f_{i,2}x=P_i$ (note that the $f_{i,1}$ and $f_{i,2}$ are here not unique, and could also be chosen without common factor). The corresponding embedding $\Gamma = \p^1 \setminus \Lambda \to\A^3$ is given by
$$[x:y]\mapsto\left(\begin{array}{l}\vspace{0.1cm}
\displaystyle{\dfrac{1}{r}\sum\limits_{i=1}^r \dfrac{a_i x^n-b_i y^n}{a_i x^n+b_i y^n}} \\
\vspace{0.1cm}
\displaystyle{\dfrac{1}{r}\sum\limits_{i=1}^r \dfrac{-2b_i xy^{n-1}}{a_i x^n+b_i y^n}}\\
\vspace{0.1cm}
\displaystyle{\dfrac{1}{r}\sum\limits_{i=1}^r \dfrac{2a_i x^{n-1}y}{a_i x^n+b_i y^n}}
\end{array}\right).$$

$2)$ In the dihedral case, we can assume that $H$ is generated by the maps $[x:y]\mapsto [\xi_n x:y]$ and $[x:y]\mapsto [y:x]$. So $G$ is generated by $\left( \begin{array}{cc} \zeta & 0\\ 0 & \zeta^{-1}  \end{array} \right)$ and $\left( \begin{array}{cc} 0 & \im \\ \im & 0 \end{array} \right)$, where $\im$ denotes the imaginary unit $\sqrt{-1}$.

An orbit $\Lambda_i$ of~$H$ is given by the zero set of~$p_i=a_i (x^{2n}+y^{2n})+2 b_i x^ny^n$ for some $(a_i,b_i)\in \C^2 \setminus \{(0,0)\}$ and we thus get
$$P_i=(p_i)^2\in \mathcal{O}(\A^2)^G$$
and
$$(f_{i,1},f_{i,2})=\left(y^{n-1}( b_ix^n +a_iy^{n})p_i,-x^{n-1}(a_ix^{n}+b_iy^n)p_i\right)\in \End(\A^2)^G$$
which satisfy $f_{i,1}y-f_{i,2}x=P_i$ (note that $P_i=p_i$ is also possible if $n$ is even, and that as before the polynomials $f_{i,1},f_{i,2}$ are  not unique, and could also be chosen without common factor).
This leads to the embedding $\Gamma = \p^1 \setminus \Lambda \to\A^3$ defined by
\[ [x:y]\mapsto \left(\begin{array}{l}\vspace{0.1cm}
\displaystyle{\frac{1}{r}\sum_{i=1}^r \frac{a_i(x^{2n}-y^{2n})}{a_i(x^{2n}+y^{2n})+2b_ix^ny^n}} \\
\vspace{0.1cm}
\displaystyle{\frac{1}{r}\sum_{i=1}^r \frac{-2xy^{n-1}(b_i x^n +a_i y^n)}{a_i(x^{2n}+y^{2n})+2 b_ix^ny^n}}\\
\vspace{0.1cm}
\displaystyle{\frac{1}{r}\sum_{i=1}^r \frac{2x^{n-1}y(a_i x^n+b_i y^n)}{a_i(x^{2n}+y^{2n})+2 b_ix^ny^n}}
\end{array}\right).\]

$3)$ In the case of the tetrahedral group, we can assume that $H\cong {\mathfrak A}_4$ is generated by the maps $[x:y]\mapsto [\im (x+y):x-y]$ and $[x:y]\mapsto [x:-y]$. This implies that $G$ is generated by
$\frac{1}{2}\left( \begin{array}{rr} \im-1 & \im-1 \\ \im+1 & -\im-1  \end{array} \right)$ and $\left( \begin{array}{rr}-\im  & 0 \\ 0 & \im \end{array} \right)$.
An orbit $\Lambda_i$ of~$H$ is given by the zero set of  $$p_i=6a_i (x^5y-xy^5)^2+b_i (x^4+y^4)(x^8+y^8-34x^4y^4),$$ for some $(a_i,b_i)\in \C^2 \setminus \{(0,0)\}$.
We thus get
$$\begin{array}{rcl}
P_i&=&p_i\in \mathcal{O}(\A^2)^G\\
f_{i,1}&=&a_i(x^{10}y-6x^6y^5+5x^2y^9)+b_i(-11x^8y^3-22x^4y^7 +y^{11})\\
f_{i,2}&=&-a_i(5x^9y^2-6 x^5y^6+xy^{10})-b_i(x^{11}-22x^7y^4-11 x^3 y^8)\end{array}$$ which satisfy $(f_{i,1},f_{i,2})\in \End(\A^2)^G$ and
$f_{i,1}y-f_{i,2}x=P_i$ as before.
This gives the embedding $\Gamma = \p^1 \setminus \Lambda \to\A^3$ defined by
$$[x:y]\mapsto \left(\begin{array}{l}\vspace{0.1cm}
\displaystyle{\frac{1}{r} \sum\limits_{i=1}^r \frac{4 a_i x^2 y^2 (x^4+y^4) (x^4-y^4)+b_i(x^{12}- 11 x^8 y^4 + 11 x^4 y^8 -y^{12})}
{6 a_i (x^5y-xy^5)^2+b_i (x^4+y^4)(x^8+y^8-34x^4y^4)}} \\
\vspace{0.1cm}
\displaystyle{\frac{1}{r}\sum\limits_{i=1}^r \frac{-2x
(a_i(x^{10}y-6x^6y^5+5x^2y^9)+b_i(-11x^8y^3-22x^4y^7 +y^{11}))
}{6 a_i (x^5y-xy^5)^2+b_i (x^4+y^4)(x^8+y^8-34x^4y^4)}}\\
\vspace{0.1cm}
\displaystyle{\frac{1}{r}\sum\limits_{i=1}^r
\frac{2y (a_i(5x^9y^2-6 x^5y^6+xy^{10})+b_i(x^{11}-22x^7y^4-11 x^3 y^8))}
{6a_i (x^5y-xy^5)^2+b_i (x^4+y^4)(x^8+y^8-34x^4y^4)}}
\end{array}\right).$$
It is also possible to describe similarly the other cases (${\mathfrak S}_4$ and ${\mathfrak A}_5$), but the formulas are even more intricate.

\bibliography{biblio}
\bibliographystyle{mrl}

\end{document}